\documentclass[titlepage]{article}
\usepackage[affil-it]{authblk}
\usepackage[utf8]{inputenc}

\usepackage{amsmath,bm}
\usepackage{amsfonts}
\usepackage{amssymb}
\usepackage{hyperref}
\usepackage{mathrsfs}
\usepackage[margin=1in]{geometry}
\usepackage[inline]{enumitem}
\usepackage{caption}
\usepackage{wrapfig}
\usepackage{exscale,relsize}
\usepackage{tikz}
\usetikzlibrary{matrix,calc,arrows}
\small\normalsize

\newtheorem{theorem}{Theorem}[section]
\newtheorem{lemma}[theorem]{Lemma}

\newtheorem{conjecture}[theorem]{Conjecture}
\newtheorem{observation}[theorem]{Observation}

\newenvironment{proof}[1][Proof:]
{\begin{trivlist}
\item[\hskip \labelsep {\bfseries #1}]} {\qquad\hspace*{\fill}$\square$\end{trivlist}}
        {\qquad\hspace*{\fill}$\square$\end{trivlist}}

 \newcounter{nona}[theorem]
\newcounter{nonanona}[theorem]    
\renewcommand{\thenona}{\Alph{nona}}

\newenvironment{noname}{\begin{trivlist}\item[]\refstepcounter{nona}%
        {\bf (\thenona)\ \ \ }\nobreak\noindent\sl\ignorespaces}{%
        \ifvmode\smallskip\fi\end{trivlist}}

\newcommand{\bB}{\mathcal{B}}
\newcommand{\B}{\mathfrak{B}}

\renewcommand{\S}{\mathscr{S}}

\newcommand{\T}{\mathcal{T}}
\newcommand{\Z}{\mathbb{Z}}

\newcommand{\cl}{\mathrm{cl}}

\newcommand{\un}{\mathrm{UN}}
\newcommand{\add}{\mathrm{ADD}}
\newcommand{\swap}{\mathrm{SWAP}}
\newcommand{\casc}{\mathrm{CASC}}
\newcommand{\cascgood}{\mathrm{CASC}^{\mathrm{good}}}

\renewcommand{\epsilon}{\varepsilon}

\title{Girth conditions and Rota's basis conjecture}

\author[1]{Benjamin Friedman}
\author[2,a,b]{Sean McGuinness}

\affil[1]{University of British Columbia}
\affil[2] {Thompson Rivers University}
\affil[a]{Corresponding author: smcguinness@tru.ca}
\affil[b]{Research supported by NSERC discovery grant}

\begin{document}

\maketitle
\begin{abstract}
Rota's basis conjecture (RBC) states that given a collection $\bB$ of $n$ bases in a matroid $M$ of rank $n$, one can always find $n$ disjoint rainbow bases with respect to $\bB$. In this paper, we show that if $M$ has girth at least $n-o(\sqrt{n})$, and no element of $M$ belongs to more than $o(\sqrt{n})$ bases in $\bB$, then one can find at least $n - o(n)$ disjoint rainbow bases with respect to $\bB$.
This result can be seen as an extension of the work of Geelen and Humphries, who proved RBC in the case where $M$ is paving, and $\bB$ is a pairwise disjoint collection. We make extensive use of the \emph{cascade} idea introduced by Buci\'c et al.

\vspace{.2in}


\end{abstract}

\section{Introduction}

%  Given a matroid $M = (E, \I),$ we let $r(M)$ denote the rank of $M$, and we let $\E(M) = |E(M)|$. We denote the collection of bases of $M$ by $\B(M)$. The {\bf girth} of $M$, denoted $g(M)$, is the size of a smallest circuit, provided that $M$ has a circuit; otherwise $g(M)=\infty$. For a subset $X \subseteq E$, we denote the closure of $X$ in $M$ by $\cl_M(X)$. %Given a graph $G$, we let $V(G)$ and $E(G)$ denote its vertex set and edge set, respectively.  Furthermore, we let $\nu(G) = |V(G)|$ and $\E(G) = |E(G)|.$

For basic concepts and notation pertaining to matroids, we follow Oxley \cite{Oxley}. Let $M$ be a matroid of rank $n$. A {\bf base sequence} of $M$ is an $n$-tuple $\bB = (B_1, \ldots B_n) \in \B(M)^n$ of bases of $M$, where for each $i \in \{1,2, \ldots, n\}$, we think of the base $B_i$ as ``coloured'' with colour $i$. A {\bf rainbow base} ({\bf RB}) with respect to $\bB$ is a base of $M$ that contains exactly one element of each colour. Rainbow bases with respect to $\bB$ are said to be {\bf disjoint} if for each colour $c$, the representatives of colour $c$ are distinct. We let $t_M(\bB)$ denote the cardinality of a largest set of disjoint rainbow bases with respect to~$\bB$, where the subscript is dropped when $M$ is implicit.
In 1989, Rota made the following conjecture, first communicated in \cite{HuangRota}:

\begin{conjecture}[Rota's Basis Conjecture (RBC)]\label{conj_rbc}
Let $M$ be a matroid of rank $n$, and let $\bB$ be a base sequence of $M$. Then $t(\bB)=n$.
\end{conjecture}

We say that a base sequence $\bB = (B_1, \ldots B_n)$ is {\bf disjoint} if the bases $B_1, B_2, \ldots, B_n$ are pairwise disjoint. The above conjecture remains unsolved even in the case where $\bB$ is a disjoint sequence.  Due to the work of Drisko \cite{Drisko}, Glynn \cite{Glynn}, and Onn \cite{Onn} on the Alon-Tarsi conjecture \cite{Alon&Tarsi}, RBC is known to be true for $\mathbb{F}$-representable matroids of rank $p \pm 1$, where $\mathbb{F}$ is a field of characteristic $0$, and $p$ is an odd prime. See \cite{Friedman&McGuinness} for an overview of these results.
Chan \cite{Chan} and Cheung \cite{Cheung} proved RBC for matroids of rank $3$ and $4$, respectively. Wild \cite{Wild} proved the conjecture for \emph{strongly base-orderable matroids}.

One approach to RBC is to determine lower bounds on $t(\bB)$. This approach was taken by Geelen and Webb \cite{Geelen&Webb}, who proved that $t(\bB) \ge \sqrt{n-1}$, Dong and Geelen \cite{Dong&Geelen}, who showed that $t(\bB) \ge \frac n{7\log n}$, and most recently by Buci\'c et al. \cite{Bucic}, who proved that $t(\bB) \geq (1/2-o(1))n$.  We mention also an interesting recent result in \cite{Pokrovskiy} where it is shown that one can find at least $n-o(n)$ rainbow independent sets of size at least $n-o(n).$  Finding better bounds for $t(\bB)$ seems difficult.  In this paper, our goal is to show that the bound $t(\bB) \geq n - o(n)$ can be achieved for matroids of large girth and sequences of bases $\bB$ with small overlap.  

The \textbf{girth} $g(M)$ of a matroid $M$ is the length of a smallest circuit in $M$, where $g(M) = \infty$ if $M$ has no circuits.  A matroid $M$ is said to be \textbf{paving} if the girth of $M$ is at least the rank of $M$.  In \cite{Geelen&Humphries}, Geelen and Humphries proved Conjecture~\ref{conj_rbc} for paving matroids, in the case where $\bB$ is a disjoint sequence.
In the spirit of this work, we are interested in obtaining lower bounds for $t(\bB)$, for both disjoint and general base sequences, assuming large girth.  To obtain these bounds, we adapt the recent methods found in \cite{Bucic}.  Our first main theorem concerns the case when $\bB$ is disjoint:

\begin{theorem}\label{thm_main_disjoint}
Let $M$ be a matroid of rank $n$, and girth $g \geq n - \beta(n) + 1$, where $\beta:\mathbb{Z}^+ \rightarrow \mathbb{Z}^+$ is a positive integer function. If $\bB$ is a disjoint base sequence of $M$ and $n\ge 4\beta^2 + 7\beta +5$, then $t(\bB) \geq n-4 \beta(n)^2-7\beta(n) -4$. 
\end{theorem}

%\begin{theorem}\label{thm_main_overlapping}
%Let $M$ be a matroid of rank $n$, and girth $g \geq n - \beta(n) + 1$, where $\beta:\mathbb{Z}^+ \rightarrow \mathbb{Z}^+$ is a positive integer function. Suppose that $\kappa: \mathbb{Z}^+ \rightarrow %\mathbb{Z}^+$ is a positive integer function, and $\bB$ is a $\kappa$-overlapping base sequence of $M$. 
%If $\alpha: \mathbb{Z}_+ \rightarrow \mathbb{Z}_+$ is a positive integer function such that
%\[
%\sqrt{\alpha(n) - \beta(n)} \ge \frac {\kappa(n) \cdot n}{n-\alpha(n)} + 2\beta(n) +1,
%\]
%then $t(\bB) \geq n- \alpha(n) -2$.
%\end{theorem}
Given a positive integer function $\kappa: \mathbb{Z}^+ \rightarrow \mathbb{Z}^+$, we say that a base sequence $\bB$ of a rank $n$ matroid $M$ is $\pmb{\kappa}${\bf -overlapping} if no matroid element $e \in E(M)$ is contained in more than $\kappa(n)$ of the bases in $\bB$.
In this paper, we show that if 
%$\beta \sim o(n^{\frac 12})$ and $\kappa(n) \sim o(n^{\frac 12})$,  and 
$M$ is a $o(\sqrt{n})$-overlapping matroid of rank $n$ with girth $g \ge n - o(\sqrt{n}),$ then $t(\bB) \ge n - o(n).$  More specifically,  we prove the following theorem.

\begin{theorem}\label{the_main_overlapping}
Suppose $\beta$ and  $\kappa$ are such that $\beta(n) \sim o(\sqrt{n})$ and $\kappa(n) \sim o(\sqrt{n})$.
Let $M$ be a matroid of rank $n$ and girth $g \geq n - \beta(n) + 1$. If $\bB$ is a $\kappa$-overlapping base sequence of $M$ and $n > 2((2\kappa(n) + 2\beta(n) +1)^2 + \beta(n))$, then $t(\bB) \ge n - (2\kappa(n) + 2\beta(n) +1)^2 - \beta(n) -2$. 
\end{theorem}

\section{Roots and cascades}

For a positive integer $k$, we will let $[k]$ denote the set $\{ 1, \dots ,k \}.$  Throughout this section, we will assume that $M$ is a matroid of rank $n$ and girth $g \geq n-\beta(n)+1$, where $\beta: \Z_+ \rightarrow \{ 0 \} \cup [n]$. We will further suppose that $\bB = (B_1, B_2, \ldots, B_n)$ is a base sequence of $M$. Following \cite{Bucic}, we define $U = \cup_{c=1}^n \left(B_c \times \{c\}\right) = \left\{ (x,c) \mid x \in B_c, 1 \leq c \leq n \right\}$ to be the set of ``coloured elements" and we define $\pi_1: U \rightarrow E(M)$ and $\pi_2: U \rightarrow [n]$ to be the projections $\pi_1(x,c)=x$ and $\pi_2(x,c)=c$.
For a subset $A \subseteq U,$ we let $\underline{A}$ denote the set $\pi_1(A) \subseteq E(M)$.  

\subsection{Collections of disjoint rainbow independent sets, signatures}

We say that a subset $S \subseteq U$ is a \textbf{rainbow independent set}, or \textbf{RIS}, if $\pi_1 \big| S$ and $\pi_2 \big| S$ are injections, and the set $\underline{S}$ is independent in $M$. Note that an RIS of size $n$ corresponds to a rainbow base (RB) with respect to $\bB$.

We define $\mathbb{S}(\B)$ to be the family of all collections $\S$ of disjoint RIS's. Given $\S \in \mathbb{S}(\B)$, %we define $\mathfrak{rb}(\S)$ to be the number of RB's in $\S$.  
we let $F(\S) := \cup_{S \in \S} S$ be the set of ``used elements". Given a colour $b$, we let $\un_b(\S) = \{ (x,b)\in U \ \big| \ (x,b) \not\in F(\S) \} $ denote the set of ``unused'' coloured elements with colour $b$.

Let $\S$ be any finite collection of sets where, for all $S\in \S,$ we have $|S| \le n.$
For all $i \in [n]$, we define $\tau_i(\S) = \left| \{ S \in \S\ \big| \ |S| = i \} \right|$.  Furthermore, we define a vector $\bm{\tau}(\S) = (\tau_1(\S), \dots ,\tau_n(\S)) \in \mathbb{Z}^n$, called the {\bf signature} of $\S$.  We shall define an total order $\preccurlyeq$
on the signatures of collections in $\mathbb{S}(\bB)$  using the lexicographic ordering of vectors in $\mathbb{Z}^n.$  That is, for vectors $(a_1, \dots ,a_n), (b_1,\dots ,b_n) \in \mathbb{Z}^n,$ we have the following 
recursive definition:
\begin{itemize}
\item For $n=1$, the order $\preccurlyeq$ is just the usual order $\leq$ on $\mathbb{Z}$.
\item For $n > 1$, we have $(a_1, \dots ,a_n) \preccurlyeq (b_1,\dots ,b_n)$ if and only if $a_n < b_n,$ or $a_n = b_n$ and $(a_1, \dots ,a_{n-1}) \preccurlyeq (b_1,\dots ,b_{n-1})$ in the ordering $\preccurlyeq$ on $\mathbb{Z}^{n-1}$.  
\end{itemize}

%Using this, we define $\preccurlyeq$ on $\mathbb{S}(\bB)$ such that $\S' \preccurlyeq \S$ if and only if $\bm{\tau}(\S') \preccurlyeq \bm{\tau}(\S).$  
We say that a collection $\S \in \mathbb{S}(\bB)$ is {\bf maximal} if $\bm{\tau}(\S)$ is maximal with respect to $\preccurlyeq.$

For a collection $\S \in \mathbb{S}(\bB),$ where $\tau_n(\S) < |\S|,$  let $i^*(\S)$ be the largest size of a set in $\S$ which is not an RB.  In addition, assuming $\tau_i(\S) >0$ for some $i < i^*(\S),$ let $i^{**}(\S) = \max \{ i< i^*(\S) \ \big| \ \tau_i(\S) > 0 \}.$

For the most part, we will be looking at {\it truncated collections} in $\mathbb{S}(\bB);$ that is, collections having at most a fixed number of sets.
For a positive integer $\eta,$ let $\mathbb{S}_\eta(\bB) = \{ \S \in \mathbb{S}(\bB) \ \big| \ |\S| \le \eta \}.$  A collection $\S \in \mathbb{S}_\eta(\bB)$ is called $\bm{\eta}${\bf-maximal}
if $\bm{\tau}(\S)$ is maximal when $\preccurlyeq$ is restricted to signatures of collections in $\mathbb{S}_\eta(\bB).$

% For reasons that will be apparent later on, we will say that a collection $\S \in \mathbb{S}(\B)$ is {\bf nice} 
%\begin{itemize}
%\item[ii)] $\mathfrak{rb}(\S)$ is maximum among all collections in $\mathbb{S}(\bB)$, and
%\item[iii)] subject to i), $|S|=n-1$ for some $S \in \S$.
%\end{itemize}

%A root $(\S,S,b)$ is said to be {\it nice} if $\S$ is nice and $|S| = n-1.$

\subsection{Roots, addability and swappability}

We define a {\bf root} of $\bB$ to be a triple $(\S, S, b)$, where $\S \in \mathbb{S}(\bB),\;\; S \in \S$, and $b$ is a colour missing from $S$. Given a root $(\S, S, b)$, an element $(x,c) \in S$ is said to be $\pmb{(\S,S,b)}$-{\bf swappable} with {\bf witness} $(y,b) \in \un_b(\S)$, if the set $S - (x,c) + (y,b)$ is an RIS. The set of $(\S,S,b)$-swappable elements is denoted by $\swap(\S, S, b) \subseteq S$. An element $(x,c) \in U$ is said to be $\pmb{(\S,S,b)}$-{\bf addable} if either of the following is true:
\begin{itemize}
\item $S + (x,c)$ is an RIS, or
\item there is some $(y,b) \in \un_b(\S)$ and some $(x',c) \in S$ such that $S + (x,c) - (x',c) + (y,b)$ is an RIS.
\end{itemize}
In the former case, we say that $(x,c)$ is {\bf directly} $(\S,S,b)$-addable. In the latter case, we say that $(x,c)$ is {\bf indirectly} $(\S,S,b)$-addable with witness $(y,b).$  We denote the set of $(\S,S, b)$-addable elements by $\add(\S, S, b)$. 

Let $(\S,S,b)$ be a root and suppose $(x,c) \in \add(\S,S,b) \cap S_1,$ for some set $S_1\in \S - S.$  One can define a new root $(\S',S',c)$ as follows:   If $(x,c)$ is directly addable, then we define $T := S + (x,c).$  If $(x,c)$ is indirectly addable, and $(y,b) \in \un_b(\S)$ and $(x',c) \in S$ are elements such that $S + (x,c) - (x',c) + (y,b)$ is an RIS, then define $T:= S + (x,c) - (x',c) + (y,b)$.  Now define $S':= S_1 - (x,c).$  Then for $\S' = \S - \{ S,S_1\} + \{ S',T\},$  $(\S',S',c)$ is seen to be a new root. 
We denote the operation of transitioning from $(\S,S,b)$ to $(\S',S',c)$ by:
\[
(\S, S, b)\xrightarrow []{(x,c)}
(\S', S', c).\]
%Note that $\S'$ and $S'$ are not necessarily unique in the above operation.
Note that since there are possibly several choices for $(y,b)$ in the above operation, the set $T$ and hence also the collection $\S'$ is not necessarily unique.  We will make use of the following two lemmas, which are adaptations of lemmas appearing in \cite{Bucic}:

\begin{lemma}
\label{lem_swappable}
Let $(\S, S, b)$ be a root. If $(x',c)$ is $(\S,S,b)$-swappable with a witness $(y,b)$, then either $(y,b)$ is directly $(\S,S,b)$-addable, or for all $x \in B_c - \cl_M(\underline{S})$, the coloured element $(x,c)$ is indirectly $(\S, S,b)$-addable with witness $(y,b)$.
\end{lemma}

\begin{lemma}\label{lem_injection}
Let $S \in \S \in \mathbb{S}(\bB)$. Then for any colour $c$, there exists an injection $\phi_c: \underline{S} \rightarrow B_c$ such that for all $x \in \underline{S}$, the set $\underline{S}-x + \phi_c(x)$ is independent.
\end{lemma}

\subsection{Root cascades and cascadable elements}

\begin{figure}
%\boxed{
\begin{tikzpicture}[scale=1.75]
%Original Sets
\draw[very thick] (0,0) circle(1);
\draw[very thick] (2.25,0) circle(1);
\draw (4,-0.6) node[right=1pt] {\Huge{$\cdots$}};
\draw[very thick] (6,0) circle(1);

%Transformed Sets
\draw[very thick,dashed] (1,-0.4) arc(-90:90:1.25 and 0.4);
\draw[very thick,dashed] (1,-0.4) arc(340:296.5:1.2 and 1.1);
\draw[very thick,dashed] (0.41,-1.01) arc(0:-183:0.45 and 1);
\draw[very thick,dashed] (-0.48,-0.98) arc(-115:-230:1.1 and 1.07);
\draw[very thick,dashed] (-0.725,0.82) arc(200:336:0.8 and 0.5);
\draw[very thick,dashed] (0.725,0.77) arc(45:21.5:1.2 and 1.1);

\draw[very thick,dashed] (3.25,-0.4) arc(-90:90:1.25 and 0.4);
\draw[very thick,dashed] (3.25,-0.4) arc(340:296.5:1.2 and 1.1);
\draw[very thick,dashed] (2.66,-1.01) arc(0:-183:0.45 and 1);
\draw[very thick,dashed] (1.77,-0.98) arc(-115:-154:1.1 and 1.07);
\draw[very thick,dashed] (1.27,-0.5) arc(-90:90:1.15 and 0.5);
\draw[very thick,dashed] (1.27,0.5) arc(150:129:1.1 and 1.2);
\draw[very thick,dashed] (1.525,0.82) arc(200:336:0.8 and 0.5);
\draw[very thick,dashed] (2.975,0.77) arc(45:21.5:1.2 and 1.1);

\draw[very thick,dashed] (7,-0.4) arc(-90:90:1.25 and 0.4);
\draw[very thick,dashed] (7,-0.4) arc(340:296.5:1.2 and 1.1);
\draw[very thick,dashed] (6.41,-1.01) arc(0:-183:0.45 and 1);
\draw[very thick,dashed] (5.52,-0.98) arc(-115:-154:1.1 and 1.07);
\draw[very thick,dashed] (5.02,-0.5) arc(-90:90:1.15 and 0.5);
\draw[very thick,dashed] (5.02,0.5) arc(150:129:1.1 and 1.2);
\draw[very thick,dashed] (5.275,0.82) arc(200:336:0.8 and 0.5);
\draw[very thick,dashed] (6.725,0.77) arc(45:21.5:1.2 and 1.1);

%Elements
\draw[fill=black] (0,0.85) circle(0.04);
\node at (0,0.7) {$(x'_0,c_1)$};

\draw[fill=black] (-0.05,-1.25) circle(0.04);
\node at (-0.05,-1.4) {$(y_0,c_0)$};

\draw[fill=black] (2.25,0.85) circle(0.04);
\node at (2.25,0.7) {$(x'_1,c_2)$};

\draw[fill=black] (2.2,-1.25) circle(0.04);
\node at (2.2,-1.4) {$(y_1,c_1)$};

\draw[fill=black] (6,0.85) circle(0.04);
\node at (6,0.7) {$(x'_{\ell-1},c_{\ell})$};

\draw[fill=black] (5.95,-1.25) circle(0.04);
\node at (5.95,-1.4) {\small{$(y_{\ell-1},c_{\ell-1})$}};

\draw[fill=black] (1.7,0.1) circle(0.04);
\node at (1.7,0) {$(x_1,c_1)$};

\draw[fill=black] (3.95,0.1) circle(0.04);
\node at (3.95,0) {$(x_2,c_2)$};

\draw[fill=black] (5.45,0.1) circle(0.04);
\node at (5.55,0) {\small{$(x_{\ell-1},c_{\ell-1})$}};

\draw[fill=black] (7.7,0.1) circle(0.04);
\node at (7.7,0) {$(x_{\ell},c_{\ell})$};

%Set Labels
\node at (0,1.5) {\Large{$S_0$}};
\node at (2.25,1.5) {\Large{$S_1$}};
\node at (6,1.5) {\Large{$S_{\ell-1}$}};

%Transformed Set Labels
\node at (0.85,-1.7) {\Large{$\mu(S_0)$}};
\node at (3.1,-1.7) {\Large{$\mu(S_1)$}};
\node at (7.4,-1.7) {\large{$\mu(S_{\ell-1})+(x_{\ell},c_{\ell})$}};
\end{tikzpicture}
\caption{Visualization of a root cascade. The dashed lines indicate the new RIS's $\mu(S_i)$ being formed from the old RIS's $S_i$.}\label{fig_cascade}
\end{figure}
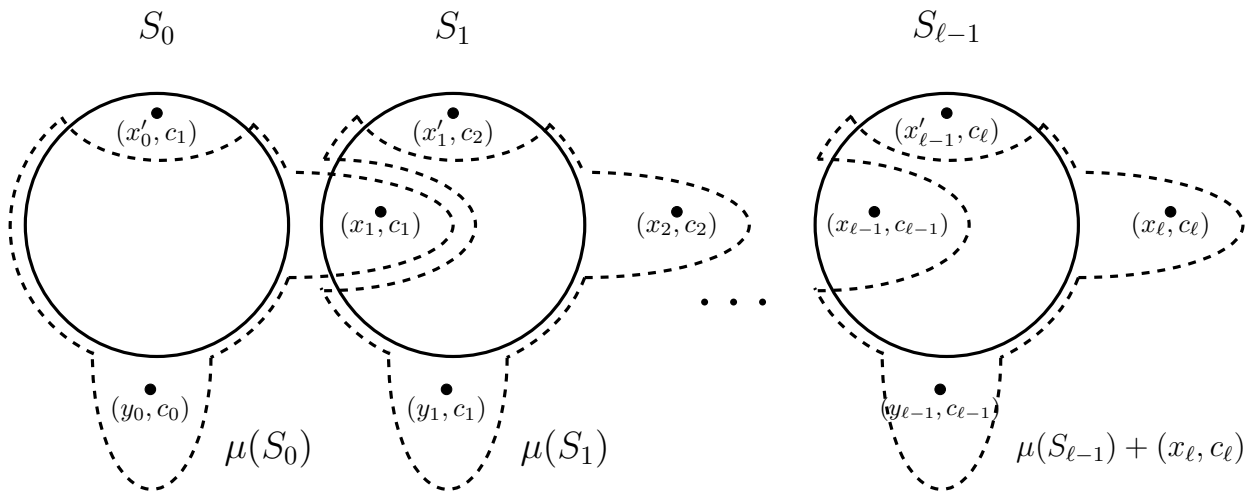

We shall adapt the notion of a \emph{cascade}, introduced in \cite{Bucic}, to that of a \emph{root cascade}.
Let $\bB$ be a base sequence of $M$ and let $\S \in \mathbb{S}(\bB)$. Let $S_0, \ldots, S_{\ell-1} \in \S$ be a sequence of distinct RIS's such that $(\S,S_0,c_0)$ is a root.  We say that $(x,c) \in U - \cup_{i=0}^{\ell-1} S_i$ is $\pmb{(\S,S_0,c_0)}${\bf-cascadable} with respect to $S_0, \ldots, S_{\ell-1} \in \S$ if there is a sequence of colours $c_0, \ldots c_{\ell-1}$ and a sequence of elements $(x_1,c_1) \in S_1, \ldots (x_{\ell-1},c_{\ell-1}) \in S_{\ell-1}$ such that: 
\begin{itemize}
\item[(C1)] We have the sequence of roots and swaps below, called a {\bf root cascade}:
\begin{equation*}
\begin{matrix}
(\S, S_0,c_0) & \xrightarrow[]{(x_1,c_1)} & (\S_1,S_1 - (x_1,c_1),c_1) & \xrightarrow[]{(x_2,c_2)} & (\S_2, S_2 - (x_2,c_2),c_2) & \xrightarrow[]{(x_3,c_3)} \cdots \\
 & & & & & \\
 & & \cdots \xrightarrow[]{(x_{\ell-1},c_{\ell-1})} & \multicolumn{2}{l}{(\S_{\ell-1}, S_{\ell-1} - (x_{\ell-1},c_{\ell-1}),c_{\ell-1})} &
\end{matrix}\label{rootseq}
\end{equation*}
and
\item[(C2)] $(x,c) \in \add(\S_{\ell -1}, S_{\ell-1} - (x_{\ell-1},c_{\ell-1}),c_{\ell-1}) .$
\end{itemize}

See Figure~\ref{fig_cascade}.  Normally, when the root $(\S,S_0,c_0)$ is implicit, we just say that $(x,c)$ is cascadable with respect to $S_0, \ldots, S_{\ell-1} \in \S$. If for some $S_\ell \in \S$, we have $(x,c) \in S_{\ell}$, then given (C2), we have $(\S_{\ell -1}, S_{\ell -1} - (x_{\ell -1}, c_{\ell -1}), c_{\ell -1}) \xrightarrow []{(x, c)} (\S_{\ell}, S_{\ell} - (x, c), c).$  We shall refer to $(\S_{\ell}, S_{\ell} - (x, c), c)$ as a {\bf root associated with} $(x, c).$  Note that this root is not necessarily unique.  Additionally, we also remark that $|\S_{\ell}| = |\S|.$

By our definition above, we note that, for $i = 1, \ldots ,\ell,$ each of the elements $(x_i,c_i)$ is cascadable with respect to $S_0, \ldots ,S_{i-1}$. 
When $\S \in \mathbb{S}_\eta(\bB)$ is an $\eta$-maximal collection and $S_0$ has size $|S_0| = i^*(\S),$ we have the following observation:

\begin{observation}\label{obs1}
Let $\S \in \mathbb{S}_\eta(\bB)$ be an $\eta$-maximal collection and assume $S_0\in \S$ has size $|S_0| = i^*(\S).$  Let $(x,c) \in U$ be $(\S,S_0,c_0)$-cascadable with respect to 
$S_0, \dots ,S_{\ell-1}$. Then:
\begin{enumerate}[label=\roman*)]
\item For $i = 1, \dots ,\ell -1,$ the set $S_i$ is an RB.
\item there exists an RB $S_{\ell} \in \S$ such that $(x,c) \in S_{\ell}$.
\end{enumerate}
\end{observation}

%Suppose $|\S|=\eta$ and $\S$ is $\eta$-maximal. In addition, suppose $|S_0| = i^*(\S)$.  Then we claim that for $i = 1, \dots ,\ell -1,$ the set $S_i$ is an RB. Furthermore, $(x_\ell, c_\ell)$ belongs to a set in $\S$ which is an RB. 
\begin{proof}
We may assume that (C1) and (C2) hold.  To prove the first assertion, suppose that $|S_i| < n$ for some $i \in [\ell-1]$.  If $i^*(\S) <n-1,$ then we see that $\tau_n(\S_i) = \tau_n(\S),$ but $i^*(\S_i) > i^*(\S),$ contradicting the maximality of $\S$. On the other hand, if $i^*(\S) = n-1,$ then $\tau_n(\S_i) = \tau_n(\S) +1,$ contradicting the maximality of $\S.$  Thus $|S_i| =n$ for all $i \in [\ell-1].$
%The first assertion can be shown inductively.  We must have $|S_1| > |S_0|$, for otherwise, $\bm{\tau}(\S) \prec \bm{\tau}(\S_1).$  Thus $S_1$ must be an RB, since $|S_0|=i^*(\S)$. Suppose now that for some $1 \le i < \ell -1$ the sets $S_j,\ j = 1, \dots ,i$ are RB's.  If $|S_{i+1}| \le |S_0|$, then $\bm{\tau}(\S) \prec \bm{\tau}(\S_{i+1}).$  Thus $|S_{i+1}| > |S_0| = i^*(\S)$ and hence is also an RB. 
To prove the second assertion,  suppose first that $(x, c) \in U - F(\S).$  Then by (C2),
$S_{\ell-1}'  = S_{\ell -1} - (x_\ell',c) + (y,c_{\ell -1}) + (x, c)$ is an RIS for some $(y,c_{\ell-1}) \in U - F(\S)$ and $(x_{\ell}', c) \in S_{\ell -1}.$ Note that $(x,c)$ cannot be directly $(\S_{\ell -1}, S_{\ell-1} - (x_{\ell-1},c_{\ell-1}),c_{\ell-1})$-addable, since $S_{\ell-1}$ is an RB.
Moreover, $\S_{\ell -1}' = \S_{\ell -1} - (S_{\ell -1}- (x_{\ell -1}, c_{\ell -1})) + S_{\ell -1}'$ is a collection for which $\bm{\tau}(\S) \prec \bm{\tau}(\S_{\ell -1}') ,$ contradicting the maximality of $\S.$  
Thus $(x, c)$ belongs to some set in $S_{\ell} \in \S - \{ S_0, \dots ,S_{\ell -1} \}$ and, again we must have $|S_{\ell}| > |S_0|;$  Otherwise, $\bm{\tau}(\S) \prec \bm{\tau}(\S_{\ell -1}').$   Consequently,  $S_{\ell}$ must also be an RB.
\end{proof}   

 We denote the set of all $(\S,S_0,c_0)$-cascadable elements with respect to $S_0, \dots, S_{\ell-1}$ by\\ $\casc_{\S,c_0}(S_0, \dots, S_{\ell-1})$.  In most cases, when the root $(\S,S_0,c_0)$ is implicit, we shall just write $\casc(S_0, \dots, S_{\ell-1})$.
 
 When (C1) holds, we note that there is a {\bf natural bijection} $\mu: \S \rightarrow \S_{\ell-1}$, defined so that
\begin{itemize}
\item $\mu(S_0)$ is the set obtained from $S_0$ by $(\S,S_0,c_0)$-addition of $(x_1,c_1)$,
\item for $1 < j < \ell-1$, the set $\mu(S_j)$ is obtained from $S_j - (x_j,c_j)$ by $(\S_j, S_j - (x_j,c_j), c_j)$-addition of $(x_{j+1},c_{j+1})$,
\item $\mu(S_{\ell-1}) = S_{\ell-1}-(x_{\ell-1},c_{\ell-1})$,
\item $\mu(S) = S$ for all $S \in \S-\{S_0, \ldots, S_{\ell-1}\}$.
\end{itemize}
The set $\mu(S)$ should be thought of as the set ``corresponding'' to $S$ following the root cascade.
 
% \section{The case where $\B$ is a disjoint collection}
 %In this section, we shall assume that $\B$ is a disjoint collection of bases.
 
\subsection{Submaximal collections}

%Let $\bm{\tau} = (t_1, t_2, \dots ,t_n)$ denote the signature of a maximal collection in $\mathbb{S}(\bB).$
Fix a positive integer $\eta \in [n]$. Observe that any two $\eta$-maximal collections have the same signature, which we will denote by $\bm{\tau}_\eta = (t_1, t_2, \dots ,t_n)$.  
Clearly $t_n \ne 0$. Assuming that $t_n < \eta$, we let $i^* = \max \{i\le n-1 \ \big| \ t_i \ne 0 \}$. That is, $i^* = i^*(\S)$, where $\S \in \mathbb{S}(\bB)$ is any $\eta$-maximal collection.

Suppose $t_{n-1} = 0.$  In this case, a collection $\S' \in \mathbb{S}_{\eta}(\bB)$ is said to be $\bm{\eta}${\bf-submaximal} if its signature $\bm{\tau}(\S')= (t_1',t_2', \dots ,t_n')$ can be obtained in the following way:  Let $\S$ be an $\eta$-maximal collection. Suppose one deletes an element from an RB in $\S$ and adds an element to another set in $\S$ of size $i^*(\S)$. Then $(t_1', \dots, t_n')$ is the signature of the resulting collection.
There are three possibilities for $(t_1',t_2', \dots ,t_n')$, depending on whether $i^* = n-2,\ i^* = n-3,$ or $i^* \le n-4$, as indicated below:

$$(t_1', \dots ,t_n') = \left\{ \begin{array}{lr}
(t_1, \dots ,t_{n-3}, t_{n-2}-1, 2, t_n-1) & \mathrm{if}\  i^* = n-2\\
(t_1, \dots ,t_{n-4}, t_{n-3}-1, 1,1,t_{n-1})    & \mathrm{if}\  i^* = n-3\\
(t_1,\dots ,t_{i^*-1}, t_{i^*}-1,1,0, \dots ,0,1,t_n-1) & \mathrm{if}\  i^* \le n-4
\end{array}\right.$$

It should be remarked that from the above, we always have that either $t_{n-1}' =1$ or $t_{n-1}' =2.$  The next lemma illustrates that if one starts with a root $(\S,S_0,c_0)$ where $|S_0| = i^*(\S)$ and $\S$ is $\eta$-maximal or $\eta$-submaximal, then all the roots in a root cascade (\ref{rootseq})
%\begin{align*}
%(\S, S_0,c_0)  \xrightarrow[]{(x_1,c_1)}& (\S_1,S_1 - (x_1,c_1),c_1)  \xrightarrow[]{(x_2,c_2)} (\S_2, S_2 - (x_2,c_2),c_2) \xrightarrow[]{(x_3,c_3)} \\
%&\cdots \xrightarrow[]{(x_{\ell-1},c_{\ell-1})} (\S_{\ell-1}, S_{\ell-1} - (x_{\ell-1},c_{\ell-1}),c_{\ell-1}) 
%\end{align*}
are either $\eta$-maximal or $\eta$-submaximal.  We leave the proof to the reader.

\begin{lemma}
Let $\S \in \mathbb{S}_{\eta}(\bB)$, where $\tau_n(\S) < |\S|,$ and let $(\S,S_0,c_0)$ be a root where $|S_0| = i^*(\S).$  Suppose $(\S,S_0,c_0) \xrightarrow[]{(x,c)} (\S', S_0',c)$.
Then we have the following:
\begin{itemize}
\item[i)]  If $\S$ is $\eta$-maximal and $t_{n-1} >0,$ then $\S'$ is $\eta$-maximal and $|S_0'| = n-1.$
\item[ii)] If $\S$ is $\eta$-maximal  and $t_{n-1}=0,$ then $\S'$ is $\eta$-submaximal and $|S_0'| = i^*(\S') = n-1.$
\item[iii)] If $\S$ is $\eta$-submaximal, then $|S_0'| =i^*(\S')$ and either $\S'$ is $\eta$-maximal and  or $\S'$ is $\eta$-submaximal.
\end{itemize}\label{lem-maxsubmax}
\end{lemma}

Corresponding to Observation \ref{obs1}, we have the following observation for $\eta$-submaximal collections:

\begin{observation}\label{obs2}
Let $\S \in \mathbb{S}_{\eta}(\bB)$ be a $\eta$-submaximal collection and let $S_0 \in \S$ where $|S_0| = n-1.$ Let $(x,c) \in U$ be $(\S,S_0,c_0)$-cascadable with respect to $S_0, \dots ,S_{\ell-1}.$  Then
$(x,c) \in S_\ell$, for some set $S_\ell \in \S - \{ S_0, \dots ,S_{\ell-1} \}.$ Moreover, for $i = 1, \dots ,\ell,$ 
$|S_i| \ge n-1,$ if $\tau_{n-1}(\S) =2$, and $|S_i| = i^{**}(\S)$ or $|S_i| =n$, if $\tau_{n-1}(\S) =1.$
\end{observation} 

\begin{proof}
We may assume that (C1) and (C2) hold.  We note that $t_{n-1} =0$ since $\S$ is $\eta$-submaximal.  Let $\bm{\tau}(\S) = (t_1', \dots ,t_n').$  Suppose $t_{n-1}' =2.$  If $|S_i| < n-1,$ for some $i \in [\ell-1],$ then we see that $\tau_n(\S_i) = t_n$ and $\tau_{n-1}(\S_i) = t_{n-1}' -1 =1 > t_{n-1} =0.$
Thus $\bm{\tau}_\eta \prec \bm{\tau}(\S_i),$ a contradiction. Suppose $t_{n-1}' =1.$  Then $t_{n-2} = t_{n-1} =0$ and $t_{n-3} > 0.$  We need only show that $|S_i| \ge i^{**}(\S).$  Suppose to the contrary that $|S_i| < i^{**}(\S)=n-2.$  Then we see that $\tau_n(\S_i) = t_n$ and $i^{**}(\S_i)= i^{**}(\S) =n-2.$  Thus have
$\bm{\tau}_\eta \prec \bm{\tau}(\S_i),$ a contradiction.

Suppose $(x,c) \in U - F(\S).$  By (C2), either $S_{\ell-1}' = S_{\ell -1} - (x_{\ell -1}, c_{\ell -1}) + (x,c)$ is an RIS, or for some 
$(y,c_{\ell-1}) \in U - F_{c_{\ell-1}}(\S_{\ell-1})$ and $(x',c) \in S_{\ell-1}$, the set $S_{\ell-1}' =  S_{\ell -1} - (x_{\ell -1}, c_{\ell -1}) - (x',c) + (y,c_{\ell-1}) + (x,c)$ is an RIS.  
In either case, let $\S_{\ell-1}'$ be the collection obtained from $\S_{\ell-1}$ by deleting the set $S_{\ell -1} - (x_{\ell -1}, c_{\ell -1})$ and adding $S_{\ell-1}'.$ Then 
$\tau_n(\S_{\ell-1}') = t_n.$  Furthermore, we see that if $t_{n-2}'=2,$ then $i^*(\S_{\ell-1}') = i^*(\S) = n-1.$  On the other hand, if $t_{n-1}' =1,$ then $i^*(\S_{\ell-1}') = i^{**}(\S) >i^*.$
Thus we see that
$\bm{\tau}_\eta \prec \bm{\tau}(\S_{\ell-1}'),$ a contradiction.  It follows that, for some set $S_\ell \in \S - \{ S_0, \dots ,S_{\ell -1} \},$ $(x,c) \in S_\ell.$
One can now use the previous arguments to show that $|S_{\ell}| \ge n-1,$ if $t_{n-1}' =2,$ and $|S_{\ell}| \ge i^{**}(\S),$ if $t_{n-1}' = 1.$ 
\end{proof}

\subsection{Finding a concentration of addable elements in a cascade}

Our primary goal in this section is to show that, under certain conditions, one can find a root $(\S, S_0,c_0)$ and sets $S_0, \dots ,S_{\ell -1}$ so that for a certain positive integer $k,$ there are at least $k$ cascadable elements with respect to $S_0, \dots ,S_{\ell-1}$ which belong to some set $S_{\ell} \in \S.$  While this is also done in \cite{Bucic}, our approach here is much simpler.
Let $\S \in \mathbb{S}(\bB)$ where $\tau_n(\S) < |\S|.$  Given a root $(\S,S_0,b)$, where $|S_0| = i^*(\S),$ we let $r(\S,S_0,b) = \max_{S' \in \S - S_0} |\add(\S,S_0,b) \cap S'|.$  Furthermore, we let $r(\S) = \max_{(\S,S_0,b)}r(\S,S_0,b)$, where the maximum is taken over all roots $(\S,S_0,b)$ such that
$|S_0| = i^*(\S).$
%the sequence $(t_1', t_2', \dots ,t_n')$ can be obtained from $(t_1, \dots ,t_n)$ by first subtracting $1$ from $t_n$ and adding it to $t_{n-1}=0$, obtaining $(t_1, \dots ,t_{n-2},1,t_n-1).$
%Next, if $i^* =2,$ then we h

%$$t_i' = \left\{ \begin{array}{lr}
%t_i^\eta &t_i\  \mathrm{if\ } i\not\in \{ n^*, n^* +1,n-1, n\}\\
%t_{i_\eta}-1\ &\mathrm{if\ } i = n^*\\
%t_{n^*+1}+1\ &\mathrm{if\ } i = n^*+1 \\
%1&\mathrm{if\ } i = n-1\\
%t_n-1&\mathrm{if\ }, i =n.
%\end{array}\right.$$  Otherwise, if $n^* = n-2,$ then
%
%$$t_i' = \left\{ \begin{array}{lr}
%t_i &t_i\  \mathrm{if\ } i\not\in \{ n-2,n-1, n\}\\
%t_{n-2}-1\ &\mathrm{if\ } i = n-2\\
%2\ &\mathrm{if\ } i = n-1\\
%t_n-1&\mathrm{if\ } i =n.
%\end{array}\right.$$

%Suppose $i^*\le n-2 .$   In this case, we say that a collection $\S \in \mathbb{S}(\bB)$ is {\bf submaximal}  if $\tau_{i^*+1}(\S) = \tau_{n-1}(\S) =1, \ 
%\tau_n(\S) = t_n-1,$ and $\tau_{i} (\S) = t_i,$ for all other value of $i.$
The following key lemma, which applies to case where $\bB$ is disjoint, can also be adapted to the case where the bases of $\bB$ are $\kappa$-overlapping.  We shall follow a similar strategy to the one used in \cite{Bucic}.  Our aim is to show that if $\S \in \mathbb{S}_\eta(\bB)$, where $\eta = n-\alpha$, 
then in the case where $t(\bB) < n-\alpha$, one can find sets $S_0, S_1, \dots ,S_\ell$ in $\S$  such that a large number of elements which are cascadable with respect to $S_0, \dots, S_{\ell-1}$ belong to the set $S_\ell.$ 
  
\begin{lemma}\label{lem-klemma}
Suppose $\bB$ is a disjoint base sequence.  Let $k$ and $\alpha$ be positive integers such that $\alpha \geq k(k+1)+\beta$ and $n \ge \alpha + 3.$  Let $\eta = n -\alpha$.  If $t(\bB) < \eta -1,$ then there is a collection $\S \in \mathbb{S}_\eta(\bB)$, which is $\eta$-maximal or $\eta$-submaximal, for which there exists a root $(\S,S_0,b)$ and distinct sets $S_0,S_1, \dots ,S_{\ell} \in \S$, $\ell \le k$, such that $|S_\ell \cap \casc_{\S,b}(S_0, \dots ,S_{\ell -1})| \ge k$.
%\begin{itemize}
%\item[i)]  There is a $\eta$-maximal collection $\S \in \mathbb{S}_\eta(\B)$  and distinct sets $S_0,S_1, \dots ,S_l \in \S$, $\ell \le k$, such that $|S_\ell \cap Q(S_0, \dots ,S_{\ell -1}| \ge k$.
%there is a nice collection $\S \in \mathbb{S}(\bB)$ with $n-\alpha$ sets, in which there are distinct RIS's $S_0, S_1, \ldots, S_{\ell}\in \S$, $\ell \leq k$, such that $|S_0| =n-1$ and $|S_{\ell} \cap Q(S_0, \dots, 
%\item[ii)] There is a $\eta$-submaximal collection $\S \in \mathbb{S}_\eta(\B)$  and distinct sets $S_0,S_1, \dots ,S_l \in \S$, $\ell \le k$, such that $|S_\ell \cap Q(S_0, \dots ,S_{\ell -1}| \ge k$.
%\end{itemize}
\end{lemma}

\begin{proof}
Since $t(\bB) < \eta -1,$ all collections in $\mathbb{S}_\eta(\bB)$ have at least two sets which are not RB's.  Among all $\eta$-maximal collections, let $\S$ be an $\eta$-maximal collection for which $r(\S)$ is maximum.  Let $\bm{\tau}(\S) = (t_1, \dots ,t_n)$.  Let $r_{\max} = r(\S)$ and let $(\S,S_0,b)$ be a root where $|S_0| = i^*(\S)$ and $r(\S,S_0,b) =r_{\max}.$   Clearly $r_{\max} \ge 1.$  Given that if $r_{\max} \ge k,$ the assertion is true, we may assume that $r_{\max} <k.$ 

By Lemma \ref{lem_injection}, for all $S \in \S \in \mathbb{S}(\bB)$, and for any colour $c$, we may assume there exists an injection $\phi_c: \underline{S} \rightarrow B_c$ such that for all $x \in \underline{S}$, the set $\underline{S}-x + \phi_c(x)$ is independent.
Let $S_0, S_1, \ldots, S_{\ell} \in \S$ be a sequence of distinct RIS's, with $\ell \le k'$, such that 

\begin{itemize}
\item[i)] $|S_\ell \cap \casc_{\S,b}(S_0, \dots ,S_{\ell -1})| = k',$ and 
\item[ii)] subject to i), $k'$ is largest possible.  
\end{itemize}

Clearly $r_{\max} \le k'.$
If $k' \ge k,$ then the assertion holds.  Thus we may assume that $k' < k.$ 
%
%It suffices to show that there is a $\eta$-maximal or $\eta$-submaximal collection $\S'$ and sets $S_0', \dots ,S_{\ell'}'$ for which $|Q(S_0', \dots ,S_{\ell'-1}') \cap S_{\ell'}'| \ge k'+1.$
Let $Q =\{(x_1,c_1), (x_2,c_2), \ldots, (x_{k'},c_{k'})\} \subseteq S_{\ell}$ denote $k'$ elements in $\casc(S_0, \dots ,S_{\ell -1})\cap S_{\ell}.$  Since $\S$ is maximal, it follows from Observation \ref{obs1} that each of the sets $S_1, \ldots, S_{\ell}$ is an RB.
For $i=1, \dots ,k',$ let $(\S_{\ell i}, S_{\ell} - (x_i,c_i), c_i)$ be a root associated with $(x_i,c_i).$  It follows by Lemma \ref{lem-maxsubmax} that $\S_{\ell i}$ is $\eta$-maximal or $\eta$-submaximal.  

\begin{noname}\label{assert1}
For $i=1, \dots ,k',$ $\add(\S_{\ell i}, S_{\ell} - (x_i,c_i), c_i)$ contains at least $n-\beta$ elements.
\end{noname}

\begin{proof}
 Let $(y,c_i) \in \un_{c_i}(\S_{\ell i}).$  If $S_{\ell} - (x_i,c_i) + (y,c_i)$ is an RB, then the collection $\S_{\ell i}'$ obtained from $\S_{\ell i}$ by replacing the set $S_{\ell} - (x_i,c_i)$ by $S_{\ell} - (x_i,c_i) + (y,c_i)$ would be such that $\bm{\tau}(\S) \prec \bm{\tau}(\S_{\ell i}')$, contradicting the maximality of $\S.$  Thus $S_{\ell} - (x_i,c_i) + (y,c_i)$ is not an RB and hence $\underline{S}_\ell - x_i +y$ contains a circuit containing $y.$
By the girth condition, there are at least $n-\beta$ elements of $\underline{S}_{\ell}-x_i$ in such a circuit, and the corresponding elements in $S_{\ell} - (x_i,c_i)$ are $(\S_{\ell i},S_{\ell}-(x_i,c_i),c_i)$-swappable with witness $(y,c_i)$. 
Thus $|\swap(\S_{\ell i},S_{\ell}-(x_i,c_i),c_i)| \ge n-\beta.$   For all colours $c,$ let $\phi_c: \underline{S}_\ell \rightarrow B_c$ be an injection such that for all $x \in \underline{S}_\ell$, the set $\underline{S}_\ell-x + \phi_c(x)$ is independent.   Then for all colours $c$, there is an element $\phi_{c}(x_i) \in B_{c}$ for which $\underline{S}_{\ell}-x_i + \phi_c(x_i)$ is independent.  By Lemma~\ref{lem_swappable}, for all $(x,c) \in \swap(\S_{\ell i}, S_{\ell} - (x_i,c_i), c_i),$ the element $(\phi_c(x_i),c)$ is $(\S_{\ell i},S_{\ell}-(x_i,c_i),c_i)$-addable.   Thus $\add(\S_{\ell i}, S_{\ell} - (x_i,c_i), c_i)$ contains at least $n-\beta$ elements.
\end{proof} 

We denote the set of elements $(\phi_c(x_i),c)\in \add(\S_{\ell i}, S_{\ell} - (x_i,c_i), c_i)$ described above by $R_i$.   The sets $R_i,\ i = 1, \dots ,k'$ are disjoint since the functions $\phi_c$ are injections.   Thus there are at least $k'(n-\beta)$ elements in $Q' = \bigcup_iR_i.$  

\begin{noname}\label{assert2}
There exists
$i \in [k'],$  such that $r(\S_{\ell i}, S_\ell - (x_i,c_i), c_i) > k'.$
\end{noname}

\begin{proof}  For $i=1, \dots ,k'$, let $\mu_i: \S \rightarrow \S_{\ell i}$ be the natural bijection. By Observations \ref{obs1} and \ref{obs2}, all elements of $R_i$ belong to sets in $\S_{\ell i}.$  By the maximality of $k'$, none of the $\eta - \ell -1 = n-\alpha-\ell-1$ sets in $\S-\{S_0, S_1, \ldots, S_{\ell}\}$ contain more than $k'$ elements of $Q'$.  Therefore, for all $S \in \S- \{S_0, S_1, \ldots, S_{\ell}\}$, we have $|S\cap Q'| \leq k'$. 
 It follows that at least $k'(n-\beta)-k'(n-\alpha-\ell-1) = k'(\alpha-\beta+ \ell +1)$ elements of $Q'$ are contained in $\cup_{i=1}^{k'} \cup_{j=0}^{\ell-1} \,\mu_i(S_j)$. Since $\alpha \geq  k(k+1) + \beta \geq \ell(k'+1)+\beta$, it follows that there are at least $k'(\alpha -\beta) \ge \ell k' (k'+1)$ such elements in $Q'.$   Thus, by averaging, for some $i,j,$ there are at least $k'+1$ such elements of $Q'$ belonging to $\mu_i(S_j)$.  That is, for some $i,j,$  $\add(\S_{\ell i},S_{\ell}-(x_i,c_i),c_i)\cap \mu_i(S_j)$ has at least $k'+1$ elements.  Consequently, $r(\S_{\ell i}, S_{\ell} - (x_i,c_i), c_i) \ge k'+1$.
\end{proof} 
 
By ({\bf \ref{assert2}}), there is an $i\in[k']$ such that $r(\S_{\ell i}, S_\ell - (x_i,c_i), c_i) > k'.$
 For such a collection, $\S_{\ell i}$ is not $\eta$-maximal; for it was, then by the choice of $r_{\max}$, $r(\S_{\ell i}, S_{\ell} - (x_i,c_i), c_i) \le r_{\max} \le k'.$   Given that $\S_{\ell i}$ is either a $\eta$-maximal or $\eta$-submaximal collection, it follows that $\S_{\ell i}$ is $\eta$-submaximal.   It now follows by Lemma \ref{lem-maxsubmax}, that $t_{n-1} =0.$
 
In light of the above, we may assume that there is a $\eta$-submaximal collection
$\S'$ where 
\begin{itemize}
\item[i)] $(\S', S_0',b')$ is a root.
\item[ii)] $|S_0'| = n-1.$ 
\item[iii)] There exist distinct sets $S_0' S_1', \dots ,S_{\ell'}' \in \S'$ where $\ell' \le k',$ and
\item[iv)] $|S_{\ell'}' \cap \casc_{\S',b'}(S_0', \dots ,S_{{\ell'} -1}')| = k'' > k'.$
\item[v)] Subject to i)- iv), $k''$ is maximum.
\end{itemize}

 If $k'' \ge k,$ then the lemma is seen to be true.  Thus we may assume $k'' < k.$  By Observation \ref{obs2}, for $i = 1, \dots ,\ell',$ either $|S_i'| \ge i^{**}(\S')$ (if $\tau_{n-1}(\S') =1$) or $|S_i'| \ge n-1$ (if $\tau_{n-1}(\S') =2$).  Let $Q = \{ (x_1',c_1'), \dots ,(x_{k''}',c_{k''}') \}$ be a subset of $k''$ elements of $S_{\ell'}' \cap Q(S_0', \dots ,S_{{\ell'} -1}').$
For $i = 1, \dots, k'',$ let $(\S_{\ell' i}',S_{\ell'}' - (x_i',c_i'),c_i')$ be a root associated with $(x_i',c_i')$ and let $R_i' \subseteq \add(\S_{\ell' i}',S_{\ell'}' - (x_i',c_i'),c_i')$ be defined similarly to $R_i.$  As before, the sets $R_i',\ i = 1, \dots ,k''$ are disjoint and $|R_i'| \ge n-\beta,\ i = 1, \dots ,k''.$  Thus there are at least $k''(n-\beta)$ elements in $Q' =  \bigcup_iR_i'$. If one of the $n-\alpha -{\ell'}-1$ sets in $\S'-\{S_0', S_1', \ldots, S_{\ell'}'\}$ has at least $k''+1$ elements in $Q'$, this would contradict the maximality of $k''.$  
Therefore, for all $S' \in \S'- \{S_0', S_1', \ldots, S_{\ell}'\}$, we have $|S'\cap Q'| \leq k''$.  By similar arguments as before, there exists an $i\in [k'']$ for which $r(\S_{\ell' i}',S_{\ell'}' - (x_i',c_i'),c_i'))\ge k''+1.$ Clearly $\S_{\ell' i}'$ is not $\eta$-maximal since $k''+1>k'.$  Thus $\S_{\ell' i}'$ is $\eta$-submaximal.
%
%It follows that at least $k''(n-\beta)-k''(n-\alpha-\ell-1) = k''(\alpha-\beta+\ell +1)$ elements of $Q'$ are contained in $\cup_{i=0}^{\ell-1} S_i'$.  Since $\alpha \geq \ell(k''+1)+\beta$, it follows that at least $\ell k'' (k''+1)$ of the elements of $Q'$ are contained in $\cup_{i=0}^{\ell-1} S_i'.$
%Thus for some $1 \le i \le k'',$ the set $R(x_i',c_i')$ has at least $\ell (k''+1)$ elements in $\bigcup_{j=0}^{\ell -1}S_j'$ and furthermore, there is some set, say $S_j'$, where $j\in\{0,1,\ldots,\ell-1\}$, such that $S_j'$ contains at least $k''+1$ of the elements in $R(x_i',c_i')$.  We may assume that $i=1.$  Let $(\S_\ell', S_\ell' - (x_1',c_1'),c_1')$ be the root associated with $(x_1',c_1').$  Then it follows that $|\add(\S_\ell',S_{\ell}'-(x_1',c_1'),c_1')\cap S_j'| \ge k'' +1$ and hence $r(\S_\ell',S_{\ell}'-(x_1',c_1'),c_1')\ge k''+1.$ 
%
%Suppose $\tau_{n-1}(\S_{\ell' i}') =2.$  Then $|S_{\ell' i}'| \ge n-1.$  If $|S_{\ell}'| =n,$ then $\S_\ell'$ is a $\eta$-submaximal collection. But given $r(\S_\ell',S_{\ell}'-(x_1',c_1'),c_1')\ge k''+1$, this would contradict v) in the choice of $\S'.$   
%If $|S_{\ell}'| = n-1,$ then $\S_{\ell}'$ is a $\eta$-maximal collection where $r(\S_\ell',S_{\ell}'-(x_1',c_1'),c_1') \ge k'' + 1 > r_{\max},$ contradicting our assumption that $r_{\max}= r(\S)$ is maximum.  Suppose $i^*(\S') =1.$  Then $|S-{\ell}'| \ge i^{**}(\S')$ and $\S_\ell'$ is seen to be an $\eta$-maximal collection.  We obtain a contradiction as before. It follows by the above that either $k' \ge k$ or $k'' \ge k.$

Given that $\S'$ is $\eta$-submaximal, there are two possibilites: either $\tau_{n-1}(\S') =1,$ or $\tau_{n-1}(\S') =2.$  Suppose $\tau_{n-1}(\S') =2.$  Then $\tau_{n-2}(\S') >0$ and hence
$i^*(\S') = n-2.$  Also, by Observation \ref{obs2}, it follows that $|S_{\ell'}'| \ge n-1.$  If $|S_{\ell'}'| = n,$ then $\S_{\ell' i}'$ is seen to be $\eta$-submaximal, and $|S_{\ell' i}' - (x_i',c_i')| = n-1.$  It would now follow by the maximality of $k''$ that $r(\S_{\ell' i}', S_{\ell'}' - (x_i',c_i'), c_i') \le k'',$ yielding a contradiction.  If $|S_{\ell'}'| = n-1,$ then $\S_{\ell' i}'$ is seen to be $\eta$-maximal, a contradiction. Suppose that $\tau_{n-1}(\S') =1.$  Then $t_{n-2} = t_{n-1} = 0$ and by Observation \ref{obs2}, $|S_{\ell'}'| = i^{**}(\S') = i^*(\S) +1$ or $|S_{\ell'}'| =n.$
When $|S_{\ell' i}'| = i^*(\S) +1,$  
$\S_{\ell' i}'$ is $\eta$-maximal .   Thus $|S_{\ell' i}'| = n$ and  $\S_{\ell' i}'$ is $\eta$-submaximal. In this case, we have that 
$r(\S_{\ell' i}', S_{\ell'}' - (x_i',c_i'), c_i') \le k'',$ yielding a contradiction.  This completes the proof.
\end{proof}

\section{Finding rainbow bases when $\bB$ is disjoint}

%Although we proved the proposition above for each positive integer $p$ such that $1 \leq p \leq 3\beta$, we are only truly interested in the case $p=3\beta$. 
Our goal in this section is to prove Theorem~\ref{thm_main_disjoint}.
The following definition applies to all base sequences $\bB$, disjoint or not.
Let $\S \in \mathbb{S}(\bB)$ and let $S,S' \in \S$ be distinct RIS's.  For elements $(x,c) \in S$ and $(x',c') \in S'$, we write $(x,c) \to (x',c')$ if $\underline{S}' - x' + x$ is independent.
%If both $(x,c) \to (x',c')$ and $(x',c') \to (x,c)$, then we write $(x,c) \leftrightarrow (x',c').$

\begin{lemma}
Suppose, for $i = 1, \dots ,k,$ there exist elements $(x_i,c_i) \in S$ and $(x_i',c_i) \in S'$ such that for all $(x_i,c_i) \in S,$ there is at least one $(x_j', c_j) \in S'$ for which $(x_i,c_i) \to (x_j',c_j).$
Then for some nonempty subset $I \subseteq \{ 1,2, \dots ,k \},$ the set $S' - \{ (x_i', c_i) \ \big| \ i \in I \} + \{ (x_i,c_i)\ \big| \ i \in I \}$ is an RIS.\label{lem-exchange} 
\end{lemma}

\begin{proof}
First, if for some $i$ we have $(x_i,c_i) \to (x_i',c_i)$, then $S' - (x_i',c_i) + (x_i,c_i)$ is an RIS. In this case, one can take $I = \{ i \}.$  Thus we may assume, for all $i$, $(x_i,c_i) \not\to (x_i',c_i).$  Given that for all $i \in \{ 1, \dots ,k \},$ there is a $j$ such that $(x_i,c_i) \to (x_j',c_j),$ there is a subset $I = \{ i_1, i_2, \dots ,i_\ell \}$ such that $(x_{i_j}, c_{i_j}) \to (x_{i_{j+1}}', c_{i_{j+1}}),$ for $j = 1, \dots, \ell -1$, and $(x_{i_\ell}, c_{i_\ell}) \to (x_{i_1}', c_{i_1}).$  Choose $I$ so that $|I| = \ell$ is minimum.  Then by minimality, we have that for all $i_{j'}, i_j \in I\backslash \{ i_\ell \},$ if $i_{j'} < i_j,$ then  $(x_{i_j},c_{i_j}) \not\to (x_{i_{j'}}',c_{i_{j'}})$.  Furthermore,
$(x_{i_\ell}, c_{i_\ell}) \not\to (x_{i_j}', c_{i_j}),$ for $j= 2, \dots ,\ell.$  Now it is seen that $S' - \{ (x_{i_j}',c_{i_j})\ \big| \ j = 1, \dots ,\ell \} + \{ (x_{i_j},c_{i_j})\ \big| \ j = 1, \dots ,\ell \}$ is an RIS.
\end{proof}

As before, let $\bm{\tau}_{\eta} = (t_1, \dots ,t_n)$ be the signature of a $\eta$-maximal collection. 

\subsection{The proof of Theorem~\ref{thm_main_disjoint}}

Let $\alpha = (2\beta +2)(2\beta +1) + \beta = 4\beta^2 + 7\beta +2$ and let $\eta = n -\alpha.$  Assume that $n\ge \alpha + 3 = 4\beta^2 + 7\beta + 5.$  Suppose $t(\bB) < \eta-2 = n - 4\beta^2 - 7\beta -4.$   It follows by Lemma \ref{lem-klemma} that there
exists
\begin{itemize}
\item[i)] a collection $\S \in \mathbb{S}_\eta(\bB),$ either $\eta$-maximal or $\eta$-submaximal, and
\item[ii)] a root $(\S,S_0,b),$ and distinct sets $S_0, S_1, \ldots, S_{\ell}$ in $\S$, where $|S_0| = i^*(\S),$ and $S_{\ell}$ contains at least $2\beta+1$ elements which are cascadable with respect to $S_0, S_1, \ldots, S_{\ell-1}$. 
\end{itemize}
 
 As before, let $\bm{\tau}_\eta = (t_1, \dots ,t_n)$ be the signature of a $\eta$-maximal collection.  Let $i^* = \max \{ 0 \le i \le n-1\ \big| \ t_i >0 \}.$  Let $(x_i', c_i),\ i = 1, \dots ,2\beta +1$ be $2\beta +1$ elements in $S_{\ell},$ as described in ii).  Each of the sets $S_1, S_2, \ldots, S_{\ell}$ is such that $|S_i| = n,$ if $\S$ is $\eta$-maximal (by Observation \ref{obs1}), and $|S_i| \ge i^{**}(\S),$ if $\S$ is $\eta$-submaximal (by Observation \ref{obs2}).  We have $\tau_n(\S) \le t(\bB) < \eta-2.$  Then there exists
$S_{\ell+1} \in \S - \{S_0, S_1, \ldots, S_{\ell}\}$ having size $|S_{\ell+1}| \leq i^*(\S)$, if $\S$ is $\eta$-maximal, and $|S_{\ell +1}| \leq n-2,$ if $\S$ is $\eta$-submaximal.   Since the girth $g\ge n-\beta +1,$ we have, for all $S\in \S,$ $|S| \ge n-\beta.$  Thus there are at least $\beta+1$ elements in $S_{\ell +1}$ having one of the colours $c_i,\  i = 1, \dots ,2\beta +1.$  Without loss of generality, we may assume $(x_i,c_i),\ i = 1, \dots ,\beta +1$ are elements in $S_{\ell +1}.$  We will show that there exists $I \subseteq \{ 1, \dots ,\beta +1\}$ such that $S_{\ell}' = S_{\ell} - \{ (x_i',c_i) \ \big| \ i\in I \} + \{ (x_i,c_i) \ \big| \ i\in I \}$ is an RIS.  Clearly, if for some $i \in [\beta+1],$$S_{\ell} - (x_i',c_i) + (x_i,c_i)$ is an RIS, then it is true.  Thus we may assume that, for $i = 1, \dots ,\beta +1,$  $S_{\ell} - (x_i',c_i) + (x_i,c_i)$ is not an RIS.
Therefore, for $i = 1, \dots ,\beta +1,$  
$\underline{S}_{\ell} + x_i$ contains a circuit (containing $x_i$).
Since $g \ge n-\beta +1$, such a circuit also contains at least $n-\beta$ elements of $\underline{S}_{\ell}.$  Thus it is seen that for $i = 1, \dots ,\beta +1,$ there is a $j\in \{ 1, \dots ,\beta+1 \}$ such that $(x_i,c_i) \to (x_j', c_j).$  It now follows by Lemma \ref{lem-exchange} that there exists $I \subseteq \{ 1, \dots ,\beta +1\}$ such that $S_{\ell}' = S_{\ell} - \{ (x_i',c_i) \ \big| \ i\in I \} + \{ (x_i,c_i) \ \big| \ i\in I \}$ is an RIS.  Without loss of generality, we may assume $1 \in I.$  Let $(\S_{\ell}, S_{\ell} - (x_1',c_1), c_1)$ be a root associated with $(x_1',c_1).$  Let $\S_{\ell}'$ be the collection obtained from $\S_{\ell}$ by replacing $S_{\ell}- (x_1',c_1)$ by $S_{\ell}'$ and $S_{\ell +1}$ by $S_{\ell + 1} -  \{ (x_i,c_i) \ \big| \ i\in I \}.$  Suppose $\S$ is $\eta$-maximal.  If $i^*(\S) = n-1,$ then it is seen that $\tau_n(\S_\ell') = t_n+1.$ If $i^*(\S) < n-1,$ then $\tau_n(\S_\ell') = t_n$ and $i^*(\S_\ell') = |S_0| + 1 = i^*(\S) +1.$ In either case, $\bm{\tau}_\eta \prec \bm{\tau}(\S_{\ell}'),$ a contradiction. 
  On the other hand, suppose $\S$ is $\eta$-submaximal. 
 Then $t_{n-1} = 0$ and $\tau_n(\S) = t_n-1.$  Now we see that $\tau_n(\S_{\ell}') = \tau_n(\S) +1 = t_n$ and $i^*(\S_\ell') > i^*.$   Thus $\bm{\tau}_{\eta} \prec \bm{\tau}(\S_{\ell}')$,  a contradiction.
We conclude that $\tau_n(\bB) \ge \eta - 2 =  n - 4\beta^2 - 7\beta -4.$

\section{The Overlapping Case}

Our goal in this section is to prove Theorem \ref{the_main_overlapping}.  A key element in the proof of Theorem \ref{thm_main_disjoint} was using the girth requirement to guarantee that roots have a large number of swappable elements.  In the case where $\bB$ is a $\kappa$-overlapping base sequence, the girth requirement no longer guarantees this.  To get around this problem, we will describe a process which transforms a {\it bad root} into a {\it good root}.
\subsection{Obtaining a good root}

Suppose that $\S \in \mathbb{S}(\bB)$, where $\bB$ is a $\kappa$-overlapping base sequence of $M$. A root $(\S, S, b)$ is {\bf good} if there exists an element $(y,b) \in \un_b(\S)$ such that $y \not\in \underline{S}$.  Note that the existence of such an element $(y,b)$ together with the girth requirement $g \ge n-\beta +1$ assures that $\swap(\S,S,b)$ will have at least $n-\beta$ elements.  If no such element $(y,b)$ exists, then the root is said to be {\bf bad}.
Using a graph, we outline a process for transforming a bad root into a good root whereby we exchange some elements in $S$ with unused elements so as to obtain a set $S'$ missing a colour $b'$.  The set $S'$ will have the convenient property that $\underline{S}' = \underline{S}.$  To describe this,
suppose $(\S, S, b)$ is a root. We form a graph, denoted $G=G(\S, S, b),$ whose vertices will consist of a base vertex, a subset of elements from $S$, and a subset of elements from $U - F(\S)$ which we call {\it terminal vertices} in $G.$

\begin{itemize}
\item Level $0$ consists of just one vertex, the base vertex denoted by the ordered pair $(O,b)$.
\item The next level of the graph, level $1$, is constructed as follows:  for all $(y,b) \in \un_b(\S)$, if $(y,c)\in S,$ for some $c,$ then $(y,c)$ becomes a vertex on level $1$ which is joined to $(O,b)$;  otherwise, if $y\not\in \underline{S}$, then $(y,b)$ becomes a (terminal) vertex in level $1$ which is joined to $(O,b).$
\item Suppose we have constructed $G$ up to level $\ell\ge 1$ and there are no terminal vertices in levels $1, \dots ,\ell.$ Then an element $(x',c') \in S$ is a vertex in level $\ell+1$ if $(x',c')$ is not already a vertex, and there exists a vertex $(x,c)$ 
in level $\ell$ for which $(x',c) \in \un_c(\S).$  In this case we join $(x,c)$ to $(x',c')$ by an edge.  An element $(y,c) \in \un_{c}(\S)$ is a terminal vertex in level $\ell +1$ if $y \not\in \underline{S}$ and there exists a vertex $(x,c)$ in level $\ell.$  We join $(x,c)$ and $(y,c)$ by an edge. 
\item The construction of the levels in $G$ stops when we complete a level containing a terminal vertex or no new vertices are added to a level.
\item For $\ell \ge 0,$ we denote by $V_{\ell}$ the set of all vertices in levels $0,\dots, \ell$; that is, $V_\ell$ is the set vertices in $G$ lying at distance at most $\ell$ from $(O,b).$
\end{itemize}

We want to guarantee that the construction of $G$ always ends with the inclusion of terminal vertices.  Asssuming $|\S| \le n-\alpha,$ the next lemma shows that this happens whenever $\alpha > \kappa.$

\begin{lemma}
Suppose $\S \in \mathbb{S}_{n-\alpha}(\bB)$ where $\alpha > \kappa$.  For all levels $\ell$ in $G$, if $V_\ell$ contains no terminal vertices, then $|V_{\ell +1}| \ge |V_\ell| \cdot \frac {\alpha}{\kappa} > |V_\ell |$.  Moreover, the graph $G$ must contain terminal vertices.
%If $V_{\ell}$ contains no terminal vertices, then $|V_{\ell}| \geq \left(\frac{\alpha}{\kappa}\right)^{\ell}$.
\label{lem-levelbound}
\end{lemma}

\begin{proof}
Assume that $V_{\ell}$ has no terminal vertices.   Then $(x',c') \in V_{\ell +1}$ is a non-terminal vertex if and only if for some $(x,c) \in V_\ell,$ $(x',c) \in \un_c(\S).$  
On the other hand, $(x',c)$ is a terminal vertex if for some $(x,c) \in V_{\ell},$ we have $(x',c) \in \un_{c'}(\S)$ and $x' \not\in \underline{S}.$  Since $\bigcup_{c \in \pi_2(V_\ell)}\un_c(\S)$ has at least $|V_\ell| \cdot \alpha$ elements, and seeing as no element of $M$ can have more than $\kappa$ colours, it follows that
$|V_{\ell +1}| \ge |V_\ell| \cdot \frac {\alpha}{\kappa} > |V_\ell |.$   To prove that $G$ must have terminal vertices, suppose to the contrary that $G$ has no terminal vertices.   Let $\ell$ be the highest level in $G.$  Then $V_\ell$ has only non-terminal vertices, and no new vertices are added to the next level.  However, by the above, there must be some vertices added to the next level since $|V_{\ell+1}| > |V_{\ell}|.$  Thus $G$ must contain terminal vertices.
\end{proof}

Assume that $\alpha > \kappa.$  Then $G$ has terminal vertices.
Let $(O,b),(x_1,c_1),(x_2,c_2),\ldots,(x_{h}, c_{h}),(x_{h+1},c_{h})$ be a shortest path from $(O,b)$ to a terminal vertex $(x_{h+1},c_h).$ Then the set 
$$S'=S - \{(x_1,c_1),(x_2,c_2),\ldots,(x_{h}, c_{h})\} + \{(x_1,b),(x_2,c_1),\ldots,(x_h,c_{h-1})\}$$ is an RIS missing the colour $b'=c_h$ and $x_{h+1} \not\in \underline{S}'.$  Note that $\underline{S}'=\underline{S}$, since the only thing changing is the colour data; this ensures that $\underline{S}'$ is indeed an independent set. Letting $\S' = \S - S + S'$, we have obtained a good root $(\S', S', b')$.
We denote this operation with
\[
(\S, S, b)\xrightarrow []{\mathrm{good}}
(\S', S', b').\]

\subsection{Good root cascades}

Let $(\S_0, S_0, c_0)$ be a root, and let $S_0, S_1, \ldots, S_k \in \S$ be a sequence of distinct RIS's. Then we say that an element $(x,c) \in U$ is $(\S,S_0,c_0)$-{\bf good-cascadable} with respect to $S_0, S_1, \ldots, S_k$ if there is a sequence of colours $c'_0, c_1, c'_1, c_2, c'_2, \ldots c_k',c_k$, and a sequence of elements $(x_1,c_1) \in S_1, (x_2,c_2) \in S_2, \ldots, (x_k,c_k) \in S_k$ such that we can perform the sequence of operations, as illustrated in Figure \ref{g_cascade}, which we call a {\bf good root cascade}.  Afterwards, we will have $(x,c) \in \add(\S_k', S_k', c_k')$. We shall simply say that $(x,c)$ is good-cascadable with respect to $S_0, S_1, \ldots, S_k$ when the root $(\S,S_0,c_0)$ is implicit.  Let $(\S_k',S_k',c_k') \xrightarrow[]{(x,c)} (\S_{k}'', S_k' - (x,c),c).$  The root $(\S_k'', S_k' - (x,c),c)$ is called a {\it root associated with} $(x,c).$

%\[
%(\S, S_0,c_0) \xrightarrow[]{\text{good}} (\S', S_0', c_0') \xrightarrow[]{(x_1,c_1)} (\S_1,S_1 - (x_1,c_1),c_1) \xrightarrow[]{\text{good}} (\S_1', S_1',c_1') \xrightarrow[]{(x_2,c_2)} (\S_2, S_2 - (x_2,c_2),c_2)$$
%$$ \cdots \xrightarrow[]{(x_k,c_k)} (\S_k, S_k - (x_k,c_k),c_k) \xrightarrow[]{\text{good}} (\S_k', S_k', c_k'),
%\] 

\begin{figure}%{R}{0.7\textwidth}
\centering
%\begin{center}
\begin{tikzpicture}
\node (R0) at (2,0) {$(\S, S_0,c_0)$};
\node (R0g) at (8,0) {$(\S', S_0', c_0')$};
\node (R1) at (2,-3) {$(\S_1,S_1 - (x_1,c_1),c_1)$};
\node (R1g) at (8,-3) {$(\S_1',S_1',c_1')$};
\node (Rk) at (2,-9) {$(\S_k,S_k - (x_k,c_k),c_k)$};
\node (Rkg) at (8,-9) {$(\S_k', S_k',c_k')$};

\node (dots1) at (2,-6)  {\Large{$\cdots$}};
\node (dots2) at (5,-6)  {\Large{$\cdots$}};
\node (dots3) at (8,-6) {\Large{$\cdots$}};

\draw [->] (R0) to node[pos=0.5,above] {good} (R0g);
\draw [->] (R1) to node[pos=0.5,above] {good} (R1g);
\draw [->] (Rk) to node[pos=0.5,above] {good} (Rkg);
%\draw [->] (Ri) to node[pos=0.5,above] {good} (Rig);
\draw[->,rounded corners=5pt] (R0g) |- (5,-1.5) -- node[pos=0.5,above] {$(x_1,c_1)$}  (2,-1.5) -> (R1);
\draw[->,rounded corners=5pt] (R1g) |- (5,-4.5) -- node[pos=0.5,above] {$(x_2,c_2)$}  (2,-4.5) -> (dots1);
\draw[->,rounded corners=5pt] (dots3) |- (5,-7.5) -- node[pos=0.5,above] {$(x_k,c_k)$}  (2,-7.5) -> (Rk);
\end{tikzpicture}
\caption{This sequence of operations, known as a good root cascade, applied to the initial root $(\S_0, S_0,c_0)$, results in a good root $(\S_k', S_k',c_k')$ such that the good-cascadable element $(x,c)$ is $(\S_k',S'_k,c'_k)$-addable.}
%\end{center}
\label{g_cascade}
\end{figure}
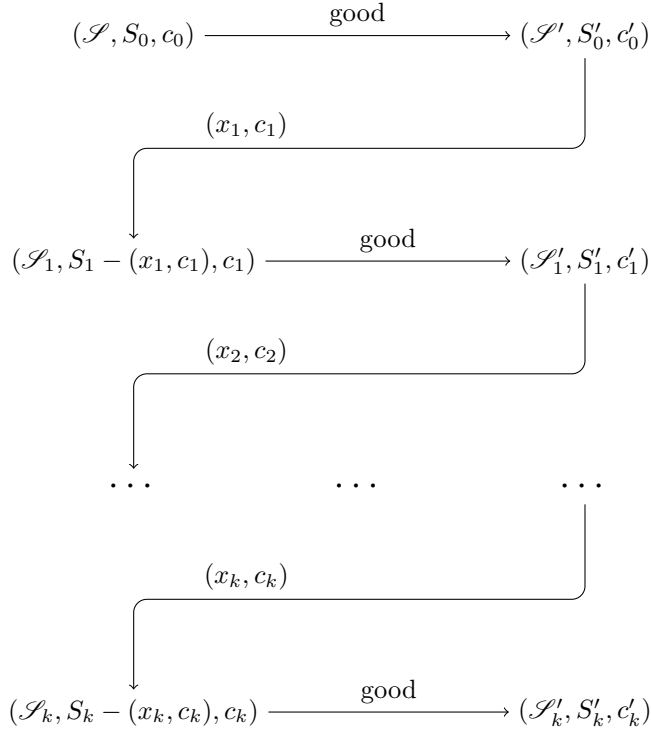

%In other words, if we first obtain a good root from the initial root, add an element, obtain a good root from the result of that addition, and continue in this manner along the entire sequence of RIS's, then the cascadable element will be addable with respect to the final root obtained. The sequence of operations in  Figure~\ref{g_cascade} followed by the $(\S, S_k', c_k')$-addition of  $(x,c)$ is called a {\bf good root cascade}.

We denote the set of all $(\S,S_0,c_0)$-good-cascadable elements with respect to $S_0, S_1, \ldots, S_k$ by $\cascgood_{\S,c_0}(S_0, S_1, \ldots, S_k)$.  When the root is implicit, we shall simply write 
$\cascgood(S_0, S_1, \ldots, S_k)$.  %There should be no confusion resulting from the two uses of this notation.

Lemma~\ref{lem-klemma} applies to good root cascades as well.  The proof of the next lemma is essentially the same as that of Lemma~\ref{lem-klemma} and we omit its proof.

\begin{lemma}
Let $k$ and $\alpha$ be positive integers where $\alpha \geq k(k+1)+\beta$ and let $\eta = n -\alpha.$  Suppose $\bB$ is a $\kappa$-overlapping sequence where $\kappa < \alpha.$     If $t(\bB) < \eta -1,$ then there is a collection $\S \in \mathbb{S}_\eta(\B)$, which is $\eta$-maximal or $\eta$-submaximal, for which there exists a root $(\S,S_0,b)$ and distinct sets $S_0,S_1, \dots ,S_l \in \S$, $\ell \le k$, such that $|S_\ell \cap \cascgood(S_0, \dots ,S_{\ell -1})| \ge k$.
\label{lem-klemma2}
\end{lemma}

In the following lemma, let $\eta = n-\alpha$ and $\bm{\tau}_\eta = (t_1, \dots ,t_n)$ be the signature for an $\eta$-maximal collection.  We shall assume that $M$ has girth $g\ge n-\beta +1$ and $\bB$ is $\kappa$-overlapping where $\alpha > \kappa.$

\begin{lemma}\label{lem-qbound}
Let $\S$ be an $\eta$-maximal or $\eta$-submaximal collection and let $S_0, \dots ,S_{k-1}$ be distinct sets in $\S$ where $|S_0| = i^*(\S).$  Suppose $|\cascgood_{\S,b}(S_0, \ldots ,S_{k-1}) \cap S|=q \ne 0$ for some set $S\in \S - \{ S_0, \ldots ,S_{k-1} \}$. Then for all $S' \in \S - \{ S_0, \dots ,S_{k-1}, S \},$ we have $|\underline{S} \cap \underline{S'}| \ge q - 2\beta$ if:
\begin{itemize}
\item[i)] $\S$ is $\eta$-maximal and $|S'| \le n-1$, or
\item[ii)] $\S$ is $\eta$-submaximal and $|S'| < \left\{ \begin{array}{lr} n-1& \mathrm{if\ } \tau_{n-1}(\S) =2\\  i^{**}(\S) & \mathrm{if}\ \tau_{n-1}(\S) =1\end{array}\right.$
\end{itemize}
\end{lemma}

\begin{proof}
Assume $Q =\cascgood(S_0, \ldots ,S_{k-1}) \cap S = \{ (x_1,c_1), (x_2,c_2), \dots ,(x_q,c_q) \}.$  We shall prove the contrapositive. Let $S' \in \S - \{ S_0, \ldots ,S_{k-1}, S \}$ where $|\underline{S} \cap \underline{S}'| < q - 2\beta.$  We shall show that neither i) nor ii) are satisfied.  Since $M$ has girth $g\ge n-\beta+1$, we have that $|S'| \ge n-\beta.$ To see this, suppose $|S'| < n-\beta.$  If for some colour $b \not\in \pi_2(S'),$
 there exists $(y,b) \in \un_b(\S)$ where $y\not\in \underline{S'},$ then $S' + (y,b)$ is seen to be an RIS.  If no such $(y,b)$ exists, then one can exchange elements in $S'$ for unused elements (as is done in transforming 
 a bad root into a good one) so as to obtain a set $S''$ where $\underline{S}'' = \underline{S}'$ and $S'' + (y,b)$ is an RIS, for some unused element $(y,b).$  In either case, replacing $S'$ by $S' + (y,b)$ in $\S$ 
 would result in a collection contradicting the maximality of $\S.$
Thus $|S'| \ge n-\beta$ and hence $S'$ has at least $q-\beta$ elements with one of the colours $c_i,\ i = 1, \dots ,q.$
Given that $|\underline{S} \cap \underline{S}'| < q - 2\beta,$ there is a subset $Q' \subseteq S'$ for which $|Q'| \ge q-\beta - (q - 2\beta -1) = \beta +1,$ $\underline{Q}' \cap \underline{S} = \emptyset,$ and
$\pi_2(Q') \subset \{ c_1, c_2, \dots ,c_q \}.$  We may assume $\{ (y_1,c_1), \dots ,(y_ {\beta+1}, c_{\beta +1}) \} \subseteq Q'.$  Since for $i = 1, \dots ,\beta +1,$ $y_i \not\in \underline{S},$ it follows that either $\underline{S} + y_i$ contains a circuit having size at least $n-\beta +1$, or $\underline{S} + y_i$ is independent.   In the former case, the circuit must contain at least one of the elements $x_i,\ i = 1, \dots ,\beta +1$ and thus $(y_i,c_i) \to (x_j, c_j),$ for some $1 \le j \le \beta +1.$  In the latter case, for $j=1, \dots ,\beta +1$, $(y_i,c_i) \to (x_j,c_j).$  It now follows by Lemma \ref{lem-exchange} that there exists $I \subseteq \{ 1,\dots ,\beta +1\}$ such that $T = S - \{ (x_i,c_i)\ \big| \ i \in I \} + \{ (y_i,c_i)\ \big| \ i \in I \}$ is an RIS.
Let $T' = S' - \{ (y_i,c_i)\ \big| \ i \in I \}$.

Without loss of generality, we may assume $1 \in I.$  Let $(\S', S - (x_1,c_1),c_1)$ be a root associated with $(x_1,c_1).$  Let $\T = \S' - \{ S', S - (x_1,c_1) \} + \{ T, T' \}.$
%By Lemma \ref{lem-maxsubmax}, $\S'$ is a $\eta$-maximal or $\eta$-submaximal collection.  
Suppose $\S$ is $\eta$-maximal.  Then Observation \ref{obs1} implies that, for $i = 1, \dots ,k-1,$ $S_i$ is an RB,  and $S$ is an RB as well.
Suppose $i^*(\S) = n-1.$  If $|S'| \le n-1,$ then it seen that $\tau_n(\T) = \tau_n(\S) + 1,$ contradicting the maximality of $\S.$   
%If $i^*(\S) = n-1,$ then Lemma \ref{lem-maxsubmax} i) implies that $\S'$ is also a $\eta$-maximal collection.  Furthermore, ones sees that $\tau_n(\T) = \tau_n(\S') +1,$ a contradiction. 
On the other hand, suppose $i^*(\S) < n-1$ and $|S'| \le n-1.$  Then it seen that $\tau_n(\T) = \tau_n(\S)$ and $i^*(\T) > i^*(\S).$  Thus
$\bm{\tau}_\eta = \bm{\tau}(\S) \prec \bm{\tau}(\T),$ a contradiction. %  Then by Lemma \ref{lem-maxsubmax} ii), $\S'$ is $\eta$-submaximal and $|S'| \le i^*(\S)$.  Now it is seen that $\tau_n(\T) = t_n$ and $i^*(\T) > i^*(\S).$  This contradicts the maximality of $\S.$ Thus $|S'| = n.$
Thus i) does not hold.

Suppose $\S$ is $\eta$-submaximal (in which case $t_{n-1} = 0$).   Suppose $\tau_{n-1}(\S) =2.$  Then by Observation \ref{obs2}, we have $|S| \ge n-1.$   If $|S'| < n-1$, then we see that $\tau_n(\T) = t_n$ and $\tau_{n-1}(\T) = 1 > t_{n-1} =0.$  Thus $\bm{\tau}_\eta \prec \tau(\T),$ a contradiction.  Suppose $\tau_{n-1}(\S) = 1.$  Then by Observation \ref{obs2}, we have $|S| \ge i^{**}(\S).$  If $|S'| < i^{**}(\S),$ then we see that $\tau_n(\T) = t_n$ and $\bm{\tau}_\eta \prec \tau(\T);$ a contradiction.  Thus ii) does not hold.
This completes the proof.
\end{proof}

\subsection{Proof of Theorem~\ref{the_main_overlapping}}

%We now prove Theorem~\ref{the_main_overlapping}.
%
For $\S \in \mathbb{S}(\bB),$ let $r^{\mathrm{good}}(\S) = \max_{(\S,S_0,b)}r(\S,S_0,b)$, where the maximum is taken over all good roots $(\S,S_0,b)$ such that
$|S_0| = i^*(\S);$ if no such roots exists for $\S$, then we define $r^{\mathrm{good}}(\S) := 0.$ 

As before, let $\eta = n-\alpha.$  We may assume that $t(\bB) < \eta.$  Among all $\eta$-maximal or $\eta$-submaximal collections $\S$, choose $\S$ so that $r^{\mathrm{good}}(\S)$ is maximum and let $r_{\max} = r^{\mathrm{good}}(\S).$  
By Lemma \ref{lem_injection}, for all $S \in \S \in \mathbb{S}(\bB)$, and for any colour $c$, we may assume there exists an injection $\phi_c: \underline{S} \rightarrow B_c$ such that for all $x \in \underline{S}$, the set $\underline{S}-x + \phi_c(x)$ is independent.  Let $S_0, \dots ,S_k$ be distinct sets in $\S$ where
\begin{itemize}
\item[i)] $|S_0| = i^*(\S).$ 
\item[ii)] $|\cascgood_{\S,b}(S_0, \dots ,S_{k-1})\cap S_k| =q$, and $k \le q.$  
\item[iii)] Subject to i) and ii), $q$ is maximum.
\end{itemize}
We see that $r_{\max} \le q.$  Let $\cascgood(S_0, \dots ,S_{k-1}) \cap S_k = \{ (x_1,c_1), \dots ,(x_q,c_q) \}.$
%
%Let $\S$ be a $\eta$-maximal or $\eta$-submaximal collection and let $S_0, \dots ,S_k$ be distinct sets in $\S$ where $|S_0| = i^*(\S).$  Suppose that 
%$Q_{\S,b}(S_0, \ldots ,S_{k-1}) \cap S_k = \{(x_1,c_1), \dots ,(x_q,c_q) \}$ where $k\le q-1.$  We may assume, that among such collections, $r(\S)$ is maximum. 
For $i=1, \dots ,q$ let $(\S_{ki}, S_k - (x_i,c_i), c_i)$ be a root associated 
with $(x_i,c_i)$ and let $(\S_{ki}, S_k - (x_i,c_i), c_i) \xrightarrow[]{\mathrm{good}} (\S_{ki}', S_{ki}', c_i').$  Observe that for $i = 1, \dots ,q,$ we have $\underline{S}_{ki}' = \underline{S}_k - x_i$ and $\S_{ki}'$ is either $\eta$-maximal or $\eta$-submaximal.   Since $(\S_{ki}', S_{ki}', c_i')$ is a good root, there is an element $(y_i,c_i') \in \un_{c_i'}(\S_{ki}')$ where $y_i \not\in \underline{S}_{ki}'.$  Since $\S_{ki}'$ is either $\eta$-maximal or $\eta$-submaximal, it follows that $S_{ki}' + (y_i, c_i')$ is not an RIS; if it was, then $(y_i,c_i') \in \add(\S_{ki}', S_{ki}',c_i'),$ contradicting Observations \ref{obs1} and \ref{obs2}.  Thus $\underline{S}_{ki}' + y_i$ contains a circuit of size at least $n-\beta +1$ (by the girth condition) and hence we see that $|\swap(\S_{ki}', S_{ki}', c_i')| \ge n-\beta.$
%Let $A_i = \add(\S_{ki}', S_{ki}', c_i'),\ i = 1, \dots ,q$ and let $A = \bigcup_iA_i.$  %Let $A_i'$ be the elements of $A_i$ which are cascadable wrt $S_0, \dots ,S_k.$ 
%Let $A' = \bigcup_iA_i'.$  We observe that $|A_i'| \ge |A_i| - k \ge |A_i| -q.$
By the above, for each colour $c$, there is an injection $\phi_c: \underline{S}_k \rightarrow B_c$ such that for all $x \in \underline{S}_k,$ the set $\underline{S}_k - x + \phi_c(x)$ is independent. 
For each $(x,c) \in \swap(\S_{ki}',S_{ki}',c_i')$, the element $\phi_c(x_i)\in B_c$ is such that $\underline{S}_{ki}' + \phi_c(x_i)$ is independent. Thus $(\phi_c(x_i),c) \in \add(\S_{ki}',S_{ki}',c_i')$.  Let $R_i$ be the set of such addable elements.  The sets $R_i,\ i = 1, \dots ,q$ are disjoint since the functions $\phi_c$ are injections.  Thus there are at least $q(n-\beta)$ elements in $R = \bigcup_iR_i.$  For each set $S \in \S - \{ S_0, \dots ,S_k \},$ let $q_S = |R\cap S|.$

\begin{noname}
If $t(\bB) < n-\alpha-2,$ then $q(n-\beta -q^2) < \kappa n + 2\beta (n-\alpha).$\label{nona1}
\end{noname}

\begin{proof}
As previously remarked, for $i=1, \dots ,q,$ the collection $\S_{ki}'$ is either $\eta$-maximal or $\eta$-submaximal.   Since $r_{\max} = r^{\mathrm{good}}(\S)$ is maximum, we have that $r^{\mathrm{good}}(\S_{ki}') \le r_{\max} \le q.$  Let $S' \in \S_{ki}'- S_{ki}'$ where
$S' \not\in \S - \{ S_0, \dots ,S_k \};$ that is, $S'$ is a set in $\S_{ki}'$ corresponding to one of the sets $S_0, \dots, S_{k-1}.$  Since $r(\S_{ki}', S_{ki}, c_i') \le r_{\max} \le q,$  we have $|R_i \cap S'| \le q.$  Given that there are $k$ such sets $S',$ and $k \le q,$ it follows that there are at most $q^2$ elements in $R_i$ belonging to such sets $S'.$  Thus for the other $\eta -k-1$ sets in $\S_{ki}'$ (excluding $S_{ki}'$ as well), which are exactly the sets in $\S - \{ S_0, \dots ,S_k \},$ we have that $$\sum_{S'\in \S - \{ S_0, \dots ,S_k\}}|R\cap S'| \ge q(n-\beta) - q\cdot q^2 = q(n-\beta -q^2).$$
Let $S' \in \S - \{ S_0, \dots ,S_k \}$ be a fixed set.  Given that $t(\bB) < \eta -2,$ we may choose $S'$ so that:

 a) $|S'| \le n-1,$ if $\S$ is $\eta$-maximal, or
 
 b) $|S'| < n-1$ if $\S$ is $\eta$-submaximal and $\tau_{n-1}(\S) =2,$ or
 
 c) $|S'| < i^{**}(\S)$ if $\S$ is $\eta$-submaximal and $\tau_{n-1}(\S) = 1.$  
 
  By Lemma \ref{lem-qbound}, we have, for all $S \in \S - \{ S_0, \dots ,S_k,S' \},$ $|\underline{S} \cap \underline{S}'| \ge q_S - 2\beta.$ Thus we have the following
\begin{align*}\kappa n &\ge \sum_{S \in \S - \{ S_0, \dots ,S_k,S' \}}|\underline{S} \cap \underline{S'}|\\ &\ge  \sum_{S \in \S - \{ S_0, \dots ,S_k,S' \}}(q_S - 2\beta) > \sum_{S \in \S - \{ S_0, \dots ,S_k,S' \}}q_S\  - \ 2\beta (n-\alpha -k)\\ &> \sum_{S\in \S - \{ S_0, \dots ,S_k, S'\}}|R\cap S|\  - \ 2\beta(n-\alpha -q)\\ &\ge q(n-\beta -q^2) - 2\beta(n-\alpha -q).
\end{align*} 
Thus $q(n-\beta-q^2) < \kappa n + 2\beta(n-\alpha).$
\end{proof}

\begin{noname}
If $\displaystyle{
\sqrt{\alpha - \beta} \ge \frac {\kappa \cdot n}{n-\alpha} + 2\beta +1,}$
then $t(\bB) \geq n- \alpha -2$.\label{nona2}
\end{noname}

\begin{proof}
Assuming $\alpha > \beta$, we have $\alpha \ge (\sqrt{\alpha -\beta}-1)\sqrt{\alpha-\beta} +\beta.$  Let $q = \sqrt{\alpha - \beta} -1.$  Suppose $t(\bB) < n- \alpha -2.$ Then it follows by Lemma \ref{lem-klemma2} that there is a collection $\S \in \mathbb{S}_\eta(\bB)$, which is $\eta$-maximal or $\eta$-submaximal, for which there exists a root $(\S,S_0,b)$ and distinct sets $S_0,S_1, \dots ,S_l \in \S$, $\ell \le q$, such that $|S_\ell \cap \cascgood(S_0, \dots ,S_{\ell -1})| \ge q$. By ({\bf \ref{nona1}}), it follows that $q(n-\beta -q^2) < \kappa n + 2\beta (n-\alpha).$
%From the above, it follows that $q(n-\beta -q^2) < \kappa n + 2\beta (n-\alpha - q).$  
However, the reader can check that when $\sqrt{\alpha -\beta} \ge \frac {\kappa n}{n-\alpha} + 2\beta +1,$ we have 
$$(\sqrt{\alpha - \beta} -1)(n-\beta - (\sqrt{\alpha -\beta} -1)^2) \ge \kappa n + 2\beta(n-\alpha).$$  This yields a contradiction.  Thus $t(\bB) \ge n- \alpha -2.$
\end{proof}

\begin{noname}
$t(\bB) \ge n - (2\kappa(n) + 2\beta(n) +1)^2 - \beta(n) -2$, for $n > 2((2\kappa(n) + 2\beta(n) +1)^2 + \beta(n)).$\label{nona3}
\end{noname}

\begin{proof}
Define $\alpha: \mathbb{Z}_+ \rightarrow \mathbb{Z}_+$ by $\alpha(n) = \left( 2\kappa(n) + 2\beta(n) + 1\right)^2 + \beta(n)$. When $n > 2\alpha,$ we have $\frac n\alpha -1 >1$ and thus we have
\begin{align*}
\alpha &\geq \left( \kappa + \frac{\kappa}{\left(\frac{n}{\alpha}\right)-1} + 2\beta + 1 \right)^2 + \beta \\
&= \left( \frac{\kappa \cdot n}{n - \alpha} + 2\beta + 1\right)^2+\beta,
\end{align*}
so that
\[
\sqrt{\alpha - \beta} \geq \frac{\kappa \cdot n}{n - \alpha} + 2\beta + 1.
\]
Thus it follows by ({\bf \ref{nona2}}) that $$t(\bB) \ge n-\alpha -2 = n - (2\kappa + 2\beta +1)^2 - \beta -2.$$
\end{proof}

\bibliographystyle{abbrv}
\bibliography{rota2}

\end{document}